\documentclass[leqno,12pt]{article} 

\usepackage{fullpage}

\usepackage{color}
\usepackage{amssymb}
\usepackage{amsfonts}
\usepackage{amscd}
\usepackage{amstext}
\usepackage{latexsym,amssymb,amsthm}
\usepackage{indentfirst}
\usepackage{amsmath}

\newcommand{\hum}{H^2_{\textrm{rad}}(\mathbb{R}^N)}

\newcommand{\hd}{W^{2,2}(\mathbb{R}^N)}

\newcommand{\avg}{\raise 0pt\hbox{$-$}\hskip -10.7pt\int}

\newtheorem{theorem}{Theorem}[]
\newtheorem{lemma}[]{Lemma}

\newtheorem{proposition}[]{Proposition}

\newtheorem{remark}[]{Remark}

\newcommand{\im}{\dis\int_{\mathbb{R}^N}}
\newcommand{\dis}{\displaystyle}

\author{\textbf{J. C. Oliveira Junior} \thanks{Corresponding author. \textit{E-mail address: jc.oliveira@uft.edu.br}}\\
\small Curso de Matem\'atica -- Universidade Federal do Tocantins \\
\small 77.824--838 --  Aragua\'ina, TO, Brazil\\
}
\date{}

\title{Ground state solution for a class of modified nonlinear fourth-order elliptic equation with sign-changing unbounded potential}

\begin{document}

\maketitle

\begin{abstract} We are concerned on the fourth-order elliptic equation 

\begin{equation}\tag{$P_\lambda$} \left\{ \begin{array}[c]{ll}
           \Delta^2 u- \Delta u + V(x)u -\lambda \Delta[\rho(u^2)]\rho'(u^2)u= f(u)\, \, \mbox{in} \, \, \mathbb{R}^N, &\\
           u\in W^{2,2}(\mathbb{R}^N),
        \end{array}
         \right.
\end{equation}
where $\Delta^2 = \Delta(\Delta)$ is the biharmonic operator, $3\leq N\leq 6$, the radially symmetric potential $V$ may change sign and $\inf_{\mathbb{R}^N}V(x)=-\infty$ is allowed. If $f$ satisfies a type of nonquadracity and monotonicity conditions and $\rho$ is a suitable smooth function, we prove, via variational approach, the existence of a radially symmetric nontrivial ground state  solution $u_\lambda$ for problem $(P_\lambda)$ for all $\lambda\geq 0$. 

\end{abstract}

\noindent\textbf{Mathematical Subject Classification MSC2010: 35J35, 35J62, 35Q35.}

\vspace{0.5cm}




\section{Introduction and statement of the main result}

The study of existence of a standing wave solution for quasilinear Schr\"odinger equation of the form
\begin{equation}\label{fisicaequa}
i\dfrac{\partial z}{\partial t} = -\Delta z + W(x)z - l(x,z) - k[\Delta \rho (|z|^2)]\rho'(|z|^2)z,
\end{equation}
where $z:\mathbb{R}\times\mathbb{R}^N\rightarrow \mathbb{C}$, $W:\mathbb{R}^N\rightarrow \mathbb{R}$ is a given potential, $k$ is a real constant and $l$ and $\rho$ are real functions, is related with several phenomena in physics. For instance, if $k\neq 0$ and for specific functions $l,\rho$, equations (\ref{fisicaequa}) appears in plasma physics and fluid mechanics \cite{KU, LSS, LS1, LZ}, in the classical and quantum theory of Heisenberg ferromagnet and magnons \cite{QC3, TH}, in dissipative quantum mechanics \cite{RWH}, in laser theory \cite{BG1, BR1} and in condensed matter theory \cite{MF1}. The particular case  $\rho(s)=s$ occurs in theory of superfluids (see \cite{KU, LSS, LWW} and the references in \cite{LPT}). The case $\rho(s)=(1+s)^{1/2}$ appears in the self-channeling of a high-power ultra short laser in matter (see \cite{BHS1, BHS2}).

If we want to find standing wave solutions for (\ref{fisicaequa}), we take $z(t, x) := \textrm{exp}(-iEt)u(x)$ with $E \in \mathbb{R}$ and $u:\mathbb{R}^N\to\mathbb{R}$ a function, which leads to consider the following elliptic equation
\begin{equation}\label{fisicaequa2}
\begin{array}{lc}
-\Delta u + V(x) u - k\Delta (\rho(u^2))\rho'(u^2) u= f_1(x,u) &  \textrm{in} \ \mathbb{R}^N.
\end{array}
\end{equation}

In the last sixteen years, many authors has dedicated to study problem (\ref{fisicaequa2}) via various methods. To the case $\rho(s)=s$, we refer to \cite{MOA}, where $V$ is radially symmetric and might change sign. In \cite{ZTZ}, under the condition that $V$ is a continuous function satisfying $\inf_{\mathbb{R}^N}V(x)>-\infty$, the authors apply the dual approach and the mountain pass theorem to obtain infinitely many solutions of the nonautonomous problem (\ref{fisicaequa2})  with $f_1(x,u)$ odd in the variable $u$. There are many works concerning problem (\ref{fisicaequa2}) with $V$ satisfying $\inf_{\mathbb{R}^N}V(x) > 0$ and $f_1$ a subcritical or critical nonlinearity (see, for example, \cite{DASI, FASZU,LWW,POMSH,SIVI, YADIN}). The zero mass potential case $V\equiv 0$ was studied in \cite{CJ} and the vanishing potential case was considered in \cite{AISOU}. We stress that the authors in \cite{DASI} used the Nehari method, showing that the infimum for the energy functional associated to problem (\ref{fisicaequa2}) on the Nehari set is achieved at some nontrivial solution; while, in \cite{MRS3}, the authors considered $\rho(s)=(1+s^2)^{1/2}$ and, exploring properties of the Pohozaev manifold, they proved the existence of a nontrivial positive solution. 

The following fourth-order elliptic equation, in its turn, 
\begin{equation}\label{fisicaequa5}
\begin{array}[c]{ll}
           \Delta^2 u- \Delta u + V(x)u = f_2(x,u)\, \, \mbox{in} \, \, \mathbb{R}^N,
         \end{array}
\end{equation}
where $\Delta^2(u)=\Delta(\Delta u)$, has become extremely relevant after Lazer and Mckenna \cite{LM4} propose to study periodic oscillations and traveling waves in a suspension bridge by the following Dirichlet problem
\begin{equation}\label{fisicaequa3}
\left\{\begin{array}{ll}
\Delta^2 u +  c\Delta u = h(x,u) &  \textrm{in} \ \Omega,\\
u|_{\partial\Omega}=\Delta|_{\partial\Omega}=0,
\end{array}\right.
\end{equation}
where $\Omega\subset\mathbb{R}^N$ is a bounded domain and $c\in\mathbb{R}$. Equations (\ref{fisicaequa3}) models several phenomena in physics, as static deflection of an elastic plate in a fluid and communication satellites, which the reader can find, for example, in \cite{CT5, HW7, ZL2} and in its references.

The problem (\ref{fisicaequa5}), on unbounded domain, has also attracted interest in recent years. In \cite{WTaZ}, the authors studied the asymptotically linear case and used a variant version of the mountain pass theorem to obtain the existence of a ground state solution. The superlinear case was investigated in \cite{YT1}, where the symmetric mountain pass theorem was employed to guarantee that infinitely many nontrivial solutions exist. In \cite{CHJo}, the authors established the existence of two solutions for problem (\ref{fisicaequa5}) without the term $\Delta u$ and the nonlinearity $f_2$ involving critical growth. 

In every paper about problems (\ref{fisicaequa2}) and (\ref{fisicaequa5}) that we mentioned until here, the condition $\inf_{\mathbb{R}^N}V(x) > -\infty$ was essential to overcome the well known lack of compactness due to unbounded domains.

When $\inf_{\mathbb{R}^N}V(x) =-\infty$ is allowed, very little is found in the literature. On problem (\ref{fisicaequa2}), we refer the reader  \cite{MJCRU}, where the authors considered $\rho(s)=s$ and employed  the mountain pass theorem and interaction with the limit problem to obtain the existence of a nontrivial solution. On problem (\ref{fisicaequa5}), up to our knowledge, there is no result considering this condition on potential $V$. Since we allow this condition occurs (see hypothesis $(V_2)$ below), this is the first article that takes it into account.

Our main goal is to show the existence of a stationary solution of a linear combination between  problems (\ref{fisicaequa2}) and (\ref{fisicaequa5}) (see \cite{CLW1}), namely,
\begin{equation}\label{p}\tag{$P_\lambda$} \left\{ \begin{array}[c]{ll}
           \Delta^2 u- \Delta u + V(x)u -\lambda \Delta[\rho(u^2)]\rho'(u^2)u= f(u)\, \, \mbox{in} \, \, \mathbb{R}^N, &\\
           u\in W^{2,2}(\mathbb{R}^N),
        \end{array}
         \right.
\end{equation}
with $3\leq N\leq 6$, $\lambda\geq 0$, $V$ satisfying a technical hypothesis on its negative part and possessing a asymptotic positive limit, the nondecreasing nonlinearity $f$ has a subcritical growth and the function $\rho$ is smooth and satisfies some delicate conditions.

Moreover, unlike the authors in \cite{CJ, MRS3} and others,  we avoid the use of any change of variables, and for this reason we have to restrict the dimension $N$ to deal with the term $\Delta[\rho(u^2)]\rho'(u^2)u$ in the dual approach. To face compactness issues, we will assume that $V$ is a radially symmetric potential, which guarantees some compact embeddings of a correct Sobolev space. We decide to work with a general function $\rho$ since our work includes, along others examples, two important one: $\rho(s)=s$ and $\rho(s)=(1+s)^{1/2}$. 

Hereafter, for $r\geq 1$, let us denote $|\cdot|_r$ the usual norm in $L^r(\mathbb{R}^N)$. Also, we will equip the space $\hd$ with the norm
$$
||u||^2 = \im(|\Delta u|^2 + |\nabla u|^2 + u^2)dx,
$$
which turns $(\hd,||\cdot||)$ a Hilbert space. On the potential $V:\mathbb{R}^N\to\mathbb{R}$, writing $V(x) = V^+(x) - V^-(x)$, where $V^\pm (x) :=\max\{\pm V(x), 0\}$, we impose the following conditions.
\begin{itemize}
\item[$(V_1)$] $V(x)=V(|x|)$ for all $x\in\mathbb{R}^N$ and $\dis\lim_{|x|\rightarrow +\infty} V(x) = V_\infty > 0$;
\item[$(V_2)$] Let $2^*=\dfrac{2N}{N-2}$.  If $S$ is the best constant to the Sobolev embedding $D^{1,2}(\mathbb{R}^N)\hookrightarrow L^{2^*}(\mathbb{R}^N)$, namely, 
$$
S=\dis\inf_{u\in D^{1,2}\setminus\{0\}}\dfrac{|\nabla u|_2^2}{|u|^2_{2^*}},
$$
then
$$
|V^-|_{N/2}  < S.
$$
\end{itemize}
The nonlinearity $f:\mathbb{R}\to\mathbb{R}$ belongs to $C^2(\mathbb{R},\mathbb{R})$ and satisfies:
\begin{itemize}
\item[$(f_1)$] $f(0)=f'(0)=0$.
\item[$(f_2)$] $f'(t)t^2-f(t)t\geq \delta |t|^p$ for all $t\in\mathbb{R}$ and for some $\delta>0$ and $4<p < 2_*:=\dfrac{2N}{N-4}$ if $N> 4$ or $p>4$ if $N=3$.
\item[$(f_3)$] If $F(t)=\dis\int_0^t f(s)ds$, then $\dfrac{1}{4}f(t)t - F(t)\geq 0$ for all $t\in\mathbb{R}$.
\item[$(f_4)$] $\dis\lim_{t\to +\infty}\dfrac{f(t)}{t^{p-1}}=m>0$.
\end{itemize}
On function $\rho: [0,+\infty) \to \mathbb{R}$, denoting by $\rho^{(i)}$ the $i$-th derivative of $\rho$, we will assume the following.
\begin{itemize}
\item[$(\rho_1)$] $\rho$ is continuous and belongs to $C^4((0,+\infty),\mathbb{R})$.
\item[$(\rho_2)$] $|\rho^{(i)}(s)|\leq C_i$ for all $s\in[0,+\infty)$ and for some $C_i>0$, where $i\in\{1,2,3,4\}$.
\item[$(\rho_3)$] $|s\rho'(s)\rho''(s)|\leq C_5$ for all $s\in [0,+\infty)$ and for some $C_5>0$.
\item[$(\rho_4)$] $\rho'(s)\geq -\sqrt{2}\rho''(s)s\geq 0$ for all $s\in[0,+\infty)$.
\item[$(\rho_5)$] $2\rho''(s) + \rho'''(s)s\leq 0$ for all $s\in[0,+\infty)$.
\end{itemize}

Though the assumptions $(\rho_1)-(\rho_5)$ are very technical, each of them is necessary to state and prove our main result obtained in this paper.
\begin{lemma}
The following functions $\rho:[0,+\infty)\to\mathbb{R}$ satisfy the assumptions $(\rho_1)-(\rho_5)$:
\begin{itemize}
\item[a.] $\rho(s)=a+bs$ with $a,b\geq 0$;
\item[b.] $\rho(s)=(1+s)^{1/2}$;
\item[c.] $\rho(s)=a+bs+(1+s)^{1/2}$;
\item[d.] $\rho(s)=(1+s)^\alpha$ with $1-1/\sqrt{2}\leq \alpha \leq 1$.
\end{itemize}
\end{lemma}
\begin{proof}
This is a straightforward calculation.
\end{proof}
\begin{remark}
Hypothesis $(V_2)$ allows potential $V$ to change signal and, in addition, to occur $\dis\inf_{\mathbb{R}^N}V(x)=-\infty$.
\end{remark}
\begin{remark}\label{rem12}
Condition $(f_2)$ ensures that $0<t\mapsto f(t)/t$ is increasing (and so $0<t\mapsto f(t)$ as well). However, we will need this assumption only one time in the format presented in $(f_2)$. Hypothesis $(f_4)$, in its turn, shows that $f(t)\to +\infty$ as $t\to +\infty$ and, consequently,
$$
\dis\lim_{t\to +\infty} \dfrac{F(t)}{t^{p}}=\dfrac{m}{p}>0.
$$ 
Thus, since $4<p$, necessarily
$$
\dis\lim_{t\to +\infty} \dfrac{F(t)}{t^{4}}=+\infty.
$$

\end{remark}

\begin{remark}\label{obs3}
Hypothesis $(\rho_4)$ implies that $\rho'(s)\geq -\sqrt{2}\rho''(s)s\geq -\rho''(s)s\geq 0$.
\end{remark}

In the sequel, we set precisely up the result obtained.
\begin{theorem}\label{main}
Under conditions $(V_1),(V_2)$,$(f_1)-(f_4)$ and $(\rho_1)-(\rho_5)$, problem $(P_\lambda)$ has a nontrivial radially symmetric ground state  solution $u_\lambda\in\hd$.
\end{theorem}
\section{Preliminaries and variational framework}

A direct consequence from conditions $(f_1)$ and $(f_4)$ is: for all $\varepsilon>0$, there is a constant $C_\varepsilon>0$ such that, if $s\in \mathbb{R}$, then 
\begin{equation}\label{<e}
|f(s)|\leq \varepsilon |s| + C_\varepsilon|s|^p \ \ \ \ \textrm{and} \ \ \ \ \ |F(s)|\leq \varepsilon|s|^2+C_\varepsilon|s|^{p+1}.
\end{equation}

To tackle the kind of problem $(P_\lambda)$ via variational approach, some difficulties appear naturally. The first that we point our is, since $V^-$ may not be a $L^\infty(\mathbb{R}^N)$-function, we need to show that the integral $\im V(x)u^2dx$ is finite for all $u\in\hd$. The following result proves it and one more important fact. 
\begin{lemma}\label{equiV}
Conditions $(V_1)$ and $(V_2)$ imply that the quadratic form
$$
u\mapsto ||u||^2_V := \im (|\Delta u|^2+|\nabla u|^2 + V(x)u^2)dx
$$
defines a norm in $\hd$, which is equivalent to the norm $\|\cdot\|$.
\end{lemma}
\begin{proof}
Let $u\in\hd$. By H\"older and Gagliardo-Nirenberg inequalities, one has
$$
\im V^-(x)u^2dx \leq |V^-|_{N/2}|u|_{2^*}^{2} \leq \dfrac{|V^-|_{N/2}}{S}\im |\nabla u|^2dx.
$$
This implies that 
\begin{equation}\label{eqV11}\begin{array}{rcl}
\im (|\Delta u|^2+|\nabla u|^2 + V(x)u^2)dx & \geq& \im (|\Delta u|^2+\left(1-\dfrac{|V^-|_{N/2}}{S}\right)|\nabla u|^2 + V^+(x)u^2)dx.\\
& \geq & \left(1-\dfrac{|V^-|_{N/2}}{S}\right)\im (|\Delta u|^2+|\nabla u|^2 + V^+(x)u^2)dx
\end{array}
\end{equation}
Once we clearly have 
$$
\im (|\Delta u|^2+|\nabla u|^2 + V(x)u^2)dx \leq \im (|\Delta u|^2+|\nabla u|^2 + V^+(x)u^2)dx,
$$
in view of $(\ref{eqV11})$, it is enough to show that the function
$$
u\mapsto \im (|\Delta u|^2+|\nabla u|^2 + V^+(x)u^2)dx
$$
defines an equivalent norm in $\hd$. The arguments in the proof of the Lemma 2.1 in \cite{FMM} may be applied to get this. The lemma is proved.
\end{proof}

The next result shows another difficulty found and how to overcome it. The variational approach to problem $(P_\lambda)$ is guaranteed as well.

\begin{proposition}\label{welldI}
Under assumptions $(V_1)$,$(V_2)$, $(f_1)$,$(f_4)$ and $(\rho_1)-(\rho_3)$, the Euler-Lagrange  functional $I_\lambda:\hd \to \mathbb{R}$ associated to problem $(P_\lambda)$ is given by 
$$
I_\lambda(u) = \dfrac{1}{2}\im (|\Delta u|^2 + |\nabla u|^2 + V(x)u^2)dx +\dfrac{\lambda}{4} \im |\nabla\rho( u^2)|^2dx - \im F(u)dx.
$$
Moreover, $I_\lambda\in C^1(\hd,\mathbb{R})$ and
$$
I_\lambda'(u)\varphi=\im(\Delta u\Delta \varphi + \nabla u\nabla \varphi + V(x)u\varphi)dx + \lambda\im \nabla (\rho(u^2))\cdot \nabla(\rho'(u^2)u\varphi) dx - \im f(u)\varphi.
$$
for all $\varphi\in\hd$.
\end{proposition}
\begin{proof}
Let us concern only with the term $\Phi(u):=\dfrac{1}{4}\im |\nabla \rho(u^2)|^2dx$. Since 
\begin{eqnarray}\label{convw}
\hd \hookrightarrow W^{1,r}(\mathbb{R}^N), \ \textrm{for}\  2\leq r \leq 2^* \ \ \textrm{and} \\
\hd\hookrightarrow L^s(\mathbb{R}^N),  \  \textrm{for} \ 2\leq s \leq 2_*,\nonumber
\end{eqnarray}
we have from hypothesis $(\rho_2)$ that, if $u\in\hd$, then
\begin{eqnarray*}
\dfrac{1}{4}\im |\nabla \rho(u^2)|^2dx  =  \im [\rho'(u^2)]^2u^2|\nabla u|^2 dx \leq C_1^2\im u^2|\nabla u|^2 dx,
\end{eqnarray*}
what implies by using H\"older inequality with exponents $p=3$ and $q=p/(p-1)=3/2$ that
$$
\dfrac{1}{4}\im |\nabla \rho(u^2)|^2dx \leq C_1^2 \left( \im u^6 dx \right)^{1/3}\left( \im |\nabla u|^{3} dx \right)^{2/3} < +\infty.
$$
So, $\Phi$ is well defined in $\hd$. Supposing, now, that $u$ is a classical solution for $(P_\lambda)$, i.e., $u\in C^4(\mathbb{R}^N,\mathbb{R})$ satisfies pointwise $(P_\lambda)$, consider $\varphi\in C_0^\infty(\mathbb{R}^N,\mathbb{R})$ and note that, by Divergence Theorem applied to the vector field $\textbf{V} = \nabla (\rho(u^2))\rho'(u^2)u\varphi$, one has
$$
\dis\int_{B_R} \textrm{div}  [\nabla (\rho(u^2))\rho'(u^2)u\varphi] dx = \dis\int_{B_R} \textrm{div} \ \textbf{V} dx = \dis\int_{\partial B_R} \textbf{V}\cdot \eta ds = 0,
$$ 
where $B_R$ is the ball centered in $0$ with radius $R>0$ large enough so that $\textrm{supp}\ \varphi \subset B_R$ and $\eta(x)$ is the normal vector in $x$ to $\partial B_R$, boundary of $B_R$. We may conclude that $\im \textrm{div}  [\nabla (\rho(u^2))\rho'(u^2)u\varphi] dx = 0$, i.e.,
$$
\im -\Delta[\rho(u^2)]\rho'(u^2)u\varphi dx = \im \nabla (\rho(u^2))\cdot \nabla(\rho'(u^2)u\varphi) dx.
$$
A simple but wearing calculation shows that $\Phi'(u)\varphi = \im \nabla (\rho(u^2))\cdot \nabla(\rho'(u^2)u\varphi) dx$ for all $\varphi\in C_0^\infty(\mathbb{R}^N,\mathbb{R})$. Let us check that $\Phi'$ is continuous. Consider $u_n\to u$ in $\hd$ and note that, if $\varphi\in\hd$ with $||\varphi||\leq 1$, one has

\begin{eqnarray}\label{c1phi}
|\Phi'(u_n)\varphi-\Phi(u)\varphi| &= &\Big|\im \nabla (\rho(u_n^2))\cdot \nabla(\rho'(u_n^2)u_n\varphi) dx - \im \nabla (\rho(u^2))\cdot \nabla(\rho'(u^2)u\varphi) dx \Big| \nonumber \\
& = & \Big| \im 4u_n^3\rho'(u_n^2)\rho''(u_n^2)|\nabla u_n|^2 \varphi dx - \im 4u^3\rho'(u^2)\rho''(u^2)|\nabla u|^2 \varphi dx \nonumber \\
& & + \im 2[\rho'(u_n^2)]^2u_n|\nabla u_n|^2\varphi dx - \im 2[\rho'(u^2)]^2u|\nabla u|^2\varphi dx  \\
& & + \im 2[\rho'(u_n^2)]^2u_n^2\nabla u_n\nabla \varphi dx -  \im 2[\rho'(u^2)]^2u^2\nabla u\nabla \varphi dx \Big|. \nonumber
\end{eqnarray}
We shall estimate each one of the three differences in (\ref{c1phi}). For the first one, set the functions $g_n:=u_n^2\rho'(u_n^2)\rho''(u_n^2)$ and $g_0:=u^2\rho'(u^2)\rho''(u^2)$. Thus, if we call
\begin{eqnarray*}
\mathcal{I}_{n}^{(1)}\varphi = \im 4u_n^3\rho'(u_n^2)\rho''(u_n^2)|\nabla u_n|^2 \varphi dx - \im 4u^3\rho'(u^2)\rho''(u^2)|\nabla u|^2 \varphi dx, 
\end{eqnarray*}
we have by H\"older inequality and hypothesis $(\rho_3)$ that
\begin{eqnarray}\label{1<123}
|\mathcal{I}_{n}^{(1)}\varphi| &=& 4 \Big| \im u_n g_n | \nabla u_n|^2 \varphi dx - \im ug_0|\nabla u|^2 \varphi dx \Big| \nonumber \\
&\leq &4 \im |u_ng_n|\left| |\nabla u_n|^2 - |\nabla u|^2\right||\varphi| dx + 4\im |u_ng_n-ug_0||\nabla u|^2|\varphi| dx \nonumber \\
&\leq & 4C_5\left(\im |u_n|^3|\varphi|^3 dx\right)^{1/3}\left( \im \left( |\nabla u_n| + |\nabla u| \right)^3 dx \right)^{1/3}\left( \im \left| \nabla u_n - \nabla u \right|^3 dx \right)^{1/3}  \\
& & + 4 \left(\im  |u_ng_n - ug_0|^6 dx \right)^{1/6} \left(\im |\varphi|^6 dx \right)^{1/6} \left(\im |\nabla u|^3 dx \right)^{2/3}. \nonumber
\end{eqnarray}
From $(\rho_1)$ and the convergence $u_n(x)\to u(x)$ a.e. $x\in\mathbb{R}^N$, one has $g_n(x)\to g_0(x)$ a.e. $x\in\mathbb{R}^N$. By using this, condition $(\rho_3)$ and the embeddings in (\ref{convw}), we may apply Lebesgue Theorem to obtain
$$
\im  |u_ng_n - ug_0|^6 dx \to 0
$$
as $n\to +\infty$. This and the embeddings in (\ref{convw}) one more time transform (\ref{1<123}) in
\begin{eqnarray*}
|\mathcal{I}_{n}^{(1)}\varphi| &\leq& C \left( \im \left| \nabla u_n - \nabla u \right|^3 dx \right)^{1/3}\left(\im |\varphi|^6 dx \right)^{1/6} \\
& & + C\left(\im  |u_ng_n - ug_0|^6 dx \right)^{1/6} \left(\im |\varphi|^6 dx \right)^{1/6}  \\
&\leq& o_n(1)
\end{eqnarray*}
for some $C>0$ and uniformly on $||\varphi||\leq 1$ as $n\to +\infty$. Similar arguments may be done to prove that the other two differences in (\ref{c1phi}) also goes to zero uniformly on $||\varphi||\leq 1$ as $n\to +\infty$. This shows that $\Phi'$ is continuous and, consequently, $\Phi$ is a $C^1$ functional, as we wished to prove.
\end{proof}

By Lemma \ref{equiV} and Proposition \ref{welldI}, $I_\lambda$ is a $C^1$-functional and, for each $u\in\hd$, we may write
$$
I_\lambda(u) = \dfrac{1}{2}||u||_V^2 + \dfrac{\lambda}{4}\im |\nabla \rho(u^2)|^2dx - \im F(u)dx.
$$

We mean that a function $u\in\hd$ is a weak solution for problem $(P_\lambda)$ if, for all $\varphi\in\hd$, it holds $I_\lambda'(u)\varphi=0$. Thus, nontrivial weak solutions for problem $(P_\lambda)$ are critical points of functional $I_\lambda$ and all of them are contained in the Nehari set
$$
\mathcal{N} = \{u\in\hd\setminus\{0\}; \ I_\lambda'(u)u=0\}.
$$

To prove Theorem \ref{main}, we will use the standard arguments of minimization on Nehari manifold, showing that the infimum $m:=\dis\inf_{u\in\mathcal{N}}I_\lambda(u)$ is well defined and is achieved at some nontrivial solution for problem $(P_\lambda)$.

As already said, in this problem, the lack of compactness, when we work on unbounded domain as $\mathbb{R}^N$, will be overcome by the symmetry of the problem. We restrict our functional $I_\lambda$ on the subspace contained in $\hd$ of the radially symmetric functions, i.e., the space $\hum$ of the functions $u\in\hd$ satisfying $u(x)=u(|x|)$ for all $x\in\mathbb{R}^N$. Since the properties of this space are used only in the end of this work, we decide to develop all the survey considering the whole space $\hd$ and, as soon as these properties are needed, we point out this fact.

\section{Minimization arguments and proof of Theorem \ref{main}}

We wish minimize the functional $I_\lambda$ on the set $\mathcal{N}$. For this, we present some property of the Nehari set.
\begin{lemma}
If $u\in\hd\setminus\{0\}$, then there exists $t_u>0$ such that $t_u u\in\mathcal{N}$. In particular, $\mathcal{N}\neq \emptyset$.
\end{lemma}
\begin{proof}
For $t\geq 0$, consider the $C^1$-function $g(t)=I_\lambda(tu)$. Then, we see from hypothesis $(f_1)$ that
$$
g(t)= \geq \dfrac{t^2}{2}\left(||u||_V^2-\im \dfrac{F(tu)}{t^2}dx \right) = \dfrac{t^2}{2}\left(||u||_V^2-o_t(1) \right),
$$ 
as $t\to 0^+$. Therefore, for small values of $t>0$, $g(t)>0$. On the other side, from conditions $(f_3)$ and $(\rho_2)$, one has
\begin{eqnarray}
g(t) &=& \dfrac{t^2}{2}||u||_V^2 + t^4\im [\rho'(t^2u^2)]^2u^2|\nabla u|^2 dx - \im F(tu)dx \nonumber \\
& \leq & \dfrac{t^4}{2}\left(o_t(1) + C_1^2\im u^2|\nabla u|^2dx - \im \dfrac{F(tu)}{(tu)^4}u^4dx \right) \to -\infty,
\end{eqnarray}
as $t\to +\infty$. Thus, $g$ assume a maximum, say $t_u>0$, where $g'(t_u)=I_\lambda'(t_u u)u=0$, that is, $t_uu\in\mathcal{N}$, as we wished.
\end{proof}
\begin{lemma}\label{manif}
The set $\mathcal{N}$ is a $C^1$-manifold.
\end{lemma}
\begin{proof}
Let $J_\lambda(u)=I'_\lambda(u)u$ for $u\in\hd$. Recalling that, for all $u\in\mathcal{N}$,
$$
I_\lambda'(u)u=||u||_V^2 + \lambda\im \nabla \rho(u^2)\cdot\nabla (\rho'(u^2)u^2) dx - \im f(u)udx = 0,
$$
and since 
$$
\lambda\im \nabla \rho(u^2)\cdot\nabla (\rho'(u^2)u^2) dx = 4\lambda\im \rho'(u^2)\rho''(u^2)u^4|\nabla u|^2 dx + \lambda\im |\nabla \rho(u^2)|^2 dx,
$$
we may write
\begin{eqnarray}\label{uinN}
J_\lambda(u) = I'_\lambda(u)u &=& ||u||_V^2 +4\lambda\im \rho'(u^2)\rho''(u^2)u^4|\nabla u|^2 dx + \lambda\im |\nabla \rho(u^2)|^2 dx \nonumber \\
&& - \im f(u)udx=0.
\end{eqnarray}
Hence, for all $u\in\mathcal{N}$, we have
\begin{eqnarray}\label{imer}
J_\lambda'(u)u & = & 2||u||_V^2+ 4\lambda\left[ 2\im \rho''(u^2)\rho''(u^2)u^6|\nabla u|^2 dx + 2\im \rho'(u^2)\rho'''(u^2)u^6|\nabla u|^2 dx\right.\nonumber \\
&& \left. + 4\im \rho'(u^2)\rho''(u^2)u^4|\nabla u |^2 dx + 2\im \rho'(u^2)\rho''(u^2)u^4|\nabla u |^2 dx\right]\nonumber \\
&&  + 4\lambda\im \rho'(u^2)\rho''(u^2)u^4|\nabla u|^2 dx + \lambda\im |\nabla \rho(u^2)|^2 dx - \im (f'(u)u^2+f(u)u)dx. 
\end{eqnarray}
Let us reorganize some terms. By hypothesis $(\rho_4)$, we obtain two inequality, namely,
\begin{equation}\label{in45}
\im 2[\rho''(u^2)u^2]^2u^2|\nabla u|^2 dx \leq \im [\rho'(u^2)]^2u^2|\nabla u|^2 dx = \dfrac{1}{4}\im|\nabla\rho(u^2)|^2 dx
\end{equation}
and, since $\rho''(t)\leq 0$ and $\rho'(t)\geq 0$,
\begin{equation}\label{iq4555}
2\im \rho'(u^2)\rho''(u^2)u^4|\nabla u |^2 dx \leq \im \rho'(u^2)\rho''(u^2)u^4|\nabla u |^2 dx.
\end{equation}
Thus, by  (\ref{in45}) and (\ref{iq4555}), we get from (\ref{imer}) that
\begin{eqnarray*}
J_\lambda'(u)u & \leq & 2||u||_V^2+ 4\lambda\left[ \dfrac{1}{4}\im|\nabla\rho(u^2)|^2 dx + \im \rho'(u^2)\rho''(u^2)u^4|\nabla u |^2 dx\right. \nonumber \\
& & \left. + 2\im \rho'(u^2)u^4|\nabla u|^2(2\rho''(u^2)+\rho'''(u^2)u^2) dx \right]\nonumber \\
&&  + 4\lambda\im \rho'(u^2)\rho''(u^2)u^4|\nabla u|^2 dx + \lambda\im |\nabla \rho(u^2)|^2 dx\nonumber \\
&&- \im (f'(u)u^2+f(u)u)dx.
\end{eqnarray*}
Now, by using condition $(\rho_5)$ and (\ref{uinN}), one has
\begin{eqnarray}\label{j<01}
J_\lambda'(u)u & \leq & 2||u||_V^2+ 8\lambda \im \rho'(u^2)\rho''(u^2)u^4|\nabla u |^2 dx + 2 \lambda\im |\nabla \rho(u^2)|^2 dx\nonumber \\
&&- \im (f'(u)u^2+f(u)u)dx \nonumber \\
&=& 2\im f(u)u dx - \im (f'(u)u^2+f(u)u)dx  \\
& = &  \im (f(u)u - f'(u)u^2)dx < 0, \nonumber
\end{eqnarray}
where we used hypothesis $(f_2)$ and the fact that $u\neq 0$. Since $J$ is a $C^1$-functional (see Remark \ref{Jc2} below), we apply the Implicit Function Theorem to guarantee that $\mathcal{N}$ is a $C^1$-manifold. The lemma is proved.
\end{proof}

\begin{lemma}\label{m>0}
There hold $||u||_V^2>c_0$ for all $u\in\mathcal{N}$, for some $c_0>0$, and $m=\dis\inf_\mathcal{N}I_\lambda(u)>0$.
\end{lemma}
\begin{proof}
For any $u\in\mathcal{N}$, we have
\begin{eqnarray*}
||u||_V^2 +4\lambda\im \rho'(u^2)\rho''(u^2)u^4|\nabla u|^2 dx  + \lambda\im |\nabla \rho(u^2)|^2 dx  - \im f(u)udx=0,
\end{eqnarray*}
what yields by $(\rho_4)$ (see Remark \ref{obs3})
\begin{equation}\label{Neh}
||u||_V^2\leq ||u||_V^2 +4\lambda\im \rho'(u^2)u^2|\nabla u|^2[\rho''(u^2)u^2+\rho'(u^2)] dx  = \im f(u)udx.
\end{equation}
From (\ref{<e}), for all $\varepsilon>0$, there is a positive constant $C=C(\varepsilon)$ such that
$$
||u||_V^2 \leq \im f(u)u dx \leq \varepsilon ||u||_V^2 + C||u||^{p+1}_{V},
$$
where we used the continuous embedding $\hd\hookrightarrow L^s(\mathbb{R}^N)$ for $s=2$ and $s=p+1$. So, choosing $\varepsilon$ adequately, we can find a positive constant $c_0>0$ satisfying $0<c_0\leq ||u||_V^2$, what proves the first part of the lemma. For the second one, note that, from hypotheses $(f_3)$ and $(\rho_4)$, 
\begin{eqnarray*}
I_\lambda(u) &=& I_\lambda(u) - \dfrac{1}{4}I_\lambda'(u)u = \dfrac{1}{4}||u||_V^2 - \lambda\im \rho'(u^2)\rho''(u^2)u^4|\nabla u|^2 dx \\
& & +\im\left(\dfrac{1}{4}f(u)u - F(u) \right)dx \geq \dfrac{1}{4}||u||_V^2,
\end{eqnarray*}
which, applying the first part of the lemma, provides
$
I_\lambda(u)\geq \dfrac{c_0}{4}>0,
$
and concludes the proof.
\end{proof}

Despite there is a minimizing sequence in $\mathcal{N}$ for functional $I_\lambda$, it can not be a sequence that converges weakly for a solution $u_0\in\mathcal{N}$ for problem $(P_\lambda)$. In the next result, we will show the existence of an appropriate minimizing sequence for our purpose.
\begin{lemma}\label{psseq}
There exists a $(PS)_m$-sequence $(u_n)\subset\mathcal{N}$ for functional $I_\lambda$, i.e., a sequence $(u_n)\subset\mathcal{N}$ satisfying
$$
I_\lambda(u_n)\to m \ \ \ \ \ \textrm{and} \ \ \ \ \ I_\lambda'(u_n)\to 0
$$
as $n\to+\infty$.
\end{lemma}
\begin{proof} We will apply Ekeland's Principle (see Theorem 8.5 in \cite{Willem}). Since $f\in C^2(\mathbb{R},\mathbb{R})$ and by hypotheses on function $\rho$, it is possible to show that functional $J_\lambda(u)=I_\lambda'(u)u$ belongs to $C^2(\hd,\mathbb{R})$ (see Remark \ref{Jc2} after this proof). By Lemma \ref{manif}, $J_\lambda'(u)\neq 0$ for all $u\in\mathcal{N}$. In view of $I_\lambda\in C^1(\hd,\mathbb{R})$, by Ekeland's Principle, there exist a sequence $(u_n)\subset\mathcal{N}$ and a sequence $(\lambda_n)\subset\mathbb{R}$ such that
$$
I_\lambda(u_n)\to m \ \ \ \ \ \textrm{and} \ \ \ \ \ (I_\lambda'(u_n)+\lambda_n J_\lambda'(u_n))\to 0
$$
as $n\to +\infty$. From the inequality
$$
m + o_n(1) =I_\lambda(u_n)= I_\lambda(u_n) -\dfrac{1}{4}I_\lambda'(u_n)u_n \geq \dfrac{1}{4}||u_n||^2_V,
$$
it follows that $(u_n)$ is bounded and, consequently, 
\begin{equation}\label{j>lam}
o_n(1) = || I_\lambda'(u_n) + \lambda_n J_\lambda'(u_n)|| \geq \dfrac{|I_\lambda'(u_n)u_n + \lambda_n J_\lambda'(u_n)u_n|}{||u_n||}\geq C|\lambda_n J_\lambda'(u_n)u_n|,
\end{equation}
for some $C>0$. Observing the proof of Lemma \ref{manif} (specifically (\ref{j<01})) together with hypothesis $(f_2)$, it yields
$$
|J_\lambda'(u_n)u_n| \geq \im (f'(u_n)u_n^2 - f(u_n)u_n)dx\geq \delta|u_n|^{p+1}_{p+1}.
$$
We may apply the first part of Lemma \ref{m>0}, the boundness of $(u_n)$ and the inequalities
$$
c_0\leq ||u_n||_V^2\leq  \im f(u_n)u_n dx\leq \varepsilon |u_n|_2^2 + C_\varepsilon|u_n|^{p+1}_{p+1}
$$
to guarantee the existence of a positive constant $C>0$ such that
$$
|u_n|_{p+1}^{p+1} \geq C>0.
$$
Thus, $|J_\lambda'(u_n)u_n|\geq \delta C>0$ and, substituting this in (\ref{j>lam}), we have necessarily $\lambda_n\to 0$ as $n\to +\infty$. Finally, the arguments contained in proof of Proposition \ref{welldI} may be used to ensure that, since $(u_n)$ is bounded, so $(|J_\lambda'(u_n)\varphi|)$ is bounded uniformly on $||\varphi||_V\leq 1$. To see this, it is enough apply similar inequalities as in (\ref{1<123}) together conditions $(\rho_1)$ and $(\rho_2)$ in the expression of $J_\lambda'(u_n)\varphi$. Therefore, if $||\varphi||_V\leq 1$, then
$$
o_n(1) = ||I_\lambda'(u_n)+\lambda_nJ_\lambda'(u_n)|| \geq |I_\lambda'(u_n)\varphi + \lambda_n J_\lambda'(u_n)\varphi| \geq |I_\lambda'(u_n)\varphi| +o_n(1),
$$
and consequently $I_\lambda'(u_n)\to 0$ as $n\to +\infty$. The proof is completed.
\end{proof}
\begin{remark}\label{Jc2}
In the proof of the Lemmas \ref{manif} and \ref{psseq}, we naturally used that $J$ is a $C^1$-functional (in the end of the proof of the Lemma \ref{manif}) and that $J$ is a $C^2$-functional (in the proof of the Lemma \ref{psseq}). We decide to omit the proofs of these facts, because they are just a tedious but elementary calculations, in which the assumptions on functions $\rho$ and $f$ are employed and several arguments as in proof of the Proposition \ref{welldI} are done.
\end{remark}

We are now able to prove our main result. Before it, see that every result already done until here remains true if we change $\hd$ by $\hum$. Because of this, in the next proof, we consider $I_\lambda|_{\hum}$. Also, note that every critical point in $\hum$ for the functional $I_\lambda$ is also a critical point in $\hd$ for the same functional. This is a principle of
symmetric criticality for reflexive spaces due to de Morais Filho, do \'O and Souto (see \cite{DOS}, Section 3, Proposition 3.1).
\bigskip\\
\noindent\textbf{Proof of Theorem \ref{main}:} From Lemma \ref{psseq}, consider $(u_n)\subset\mathcal{N}$ a $(PS)_m$ sequence for functional $I_\lambda$, i.e., $I_\lambda(u_n)\to m$ and $I'_\lambda(u_n)\to 0$ as $n\to+\infty$. As already done before, we have that $(u_n)$ is bounded in $\hum$ and, hence, up to a subsequence, we get the existence of $u_\lambda\in\hum$ such that the weak convergence $u_n\rightharpoonup u_\lambda$ holds as $n\to +\infty$. By Sobolev embeddings and assumptions on functions $\rho$ and $f$, we obtain $I_\lambda'(u_\lambda)\varphi=0$ for all $\varphi\in C_0^\infty(\mathbb{R}^N)$. For density argument, we have $I'_\lambda(u_\lambda)=0$. Since $H^2_{\textrm{rad}}(\mathbb{R}^N)$ is compactly embedded in $L^{p+1}(\mathbb{R}^N)$ (see Remark \ref{remen} below), it follows from (\ref{<e}), Lemma \ref{m>0}, (\ref{Neh}) and the boundedness of $(|u_n|_2)$ that
$$
0<c_0<||u_n||_V^2 \leq \im f(u_n)u_n dx \leq \varepsilon C + C|u_\lambda|^{p+1}_{p+1} +o_n(1), 
$$
for some $C>0$, what guarantees that $u_\lambda\neq 0$. Consequently, $u_\lambda\in\mathcal{N}$ is a nontrivial weak solution for problem $(P_{\lambda})$. Let us show that it is also a ground state solution. For this, by the weak convergence $u_n\rightharpoonup u_\lambda$, the convergences $u_n(x)\to u_\lambda(x)$ and $|\nabla u_n(x)|\to|\nabla u_\lambda(x)|$ a.e.  $x\in\mathbb{R}^N$  and Fatou's Lemma, hypothesis $(f_3)$ and $(\rho_4)$ may be applied to obtain, up to a subsequence,
$$
\begin{array}{rcl}
m\leq I_\lambda(u_\lambda) &=& I_\lambda(u_\lambda) - \dfrac{1}{4}I_\lambda'(u_\lambda)u_\lambda = \dfrac{1}{4}||u_\lambda||^2_V - \lambda\im \rho'(u_\lambda^2)\rho''(u_\lambda^2)u_\lambda^4|\nabla u_\lambda|^2 dx\\
&&+ \im\left(\dfrac{1}{4}f(u_\lambda)u_\lambda-F(u_\lambda)\right)dx \\
& \leq & \dis\liminf_{n\to+\infty} \left[ \dfrac{1}{4}||u_n||^2_V - \lambda\im \rho'(u_n^2)\rho''(u_n^2)u_n^4|\nabla u_n|^2 dx\right.\\
&& \left.+ \im\left(\dfrac{1}{4}f(u_n)u_n-F(u_n)\right)dx\ \right] \\
& = &  \dis\liminf_{n\to+\infty} \left( I_\lambda(u_n) - \dfrac{1}{4}I_\lambda'(u_n)u_n\right)=m,
\end{array}
$$
that is, $I_\lambda(u_\lambda)=m$ and the proof of the theorem is completed. \qed
\begin{remark}\label{remen}
The proof of the compact embedding $H^2_{\textrm{rad}}(\mathbb{R}^N)\hookrightarrow L^s(\mathbb{R}^N)$ for $2<s<2_*$ follows directly the same steps contained in section 1.5 from \cite{Willem}, where the reader will find the proof of the well known compact embedding $H^1_{\textrm{rad}}(\mathbb{R}^N)\hookrightarrow L^r(\mathbb{R}^N)$ for $2<r<2^*$. The main ingredients are the usual Sobolev embeddings and Lion's Lemma.
\end{remark}

If $\lambda=0$, the method applied in this article works for all $N\geq 3$, and the problem
\begin{equation}\tag{$P_0$} \left\{ \begin{array}[c]{ll}
           \Delta^2 u- \Delta u + V(x)u = f(u)\, \, \mbox{in} \, \, \mathbb{R}^N, &\\
           u\in W^{2,2}(\mathbb{R}^N),
        \end{array}
         \right.
\end{equation}
weakening the hypotheses on the nonlinearity $f$ adequately, has a nontrivial radially symmetric ground state solution $u_0\in\hum$. 

\bigskip
\noindent {\sc acknowledgement}.  J. C. Oliveira Junior  would like to thank very much the  Universidade de Bras\'ilia for the so pleasant production environment, where all this work was developing. The author also thanks the reviewer for corrections and suggestions.

\end{document}